\documentclass[12pt]{amsart}
\usepackage{array,CJK}
\usepackage{curves}
\usepackage{amstext}
\usepackage{amsfonts}
\usepackage{amsmath}
\usepackage{amssymb}
\def\cal{\mathcal}

\newtheorem{theorem}{Theorem}[section]
\newtheorem{corollary}[theorem]{Corollary}
\newtheorem{lemma}[theorem]{Lemma}

\theoremstyle{definition}

\newtheorem{example}[theorem]{Example}
\theoremstyle{remark}

\title[Chaotic translations on weighted Orlicz spaces]
{Chaotic translations on weighted Orlicz spaces}
\author[C-C. Chen]{Chung-Chuan Chen}
\address{Department of Mathematics Education, National Taichung University of Education, Taichung 403, Taiwan}
\email{chungchuan@mail.ntcu.edu.tw}

\author[K-Y. Chen]{Kui-Yo Chen}
\address{Department of Mathematics, National Taiwan University, Taipei 106, Taiwan}
\email{kuiyochen1230@gmail.com}

\author[S. \"Oztop]{Serap \"Oztop}
\address{Department of Mathematics, Faculty of Science, Istanbul University, Istanbul, Turkey}
\email{oztops@istanbul.edu.tr}

\author[S. M. Tabatabaie]{Seyyed Mohammad Tabatabaie}
\address{Department of Mathematics, University of Qom, Qom, Iran}
\email{sm.tabatabaie@qom.ac.ir}

\subjclass[2010]{Primary 46E30; Secondary 54H20, 47A16}

\keywords{Chaos. Topological transitivity. Weighted Orlicz space. Locally compact group.}

\date{\today}
\begin{document}

\maketitle

\begin{abstract}
Let $G$ be a locally compact group, $w$ be a weight on $G$  and $\Phi$
be a Young function. We give some characterizations for
translation operators to be topologically transitive and chaotic on the weighted Orlicz space $L_w^\Phi(G)$.
In particular, transitivity is equivalent to the blow-up/collapse property in our case.
Moreover, the dense set of periodic elements implies transitivity automatically.
\end{abstract}

\baselineskip17pt

\section{introduction}

Linear chaos and hypercyclicity have been studied intensely during the last three decades.
We refer to these classic books \cite{bmbook,gpbook,kosticbook} on this subject.
Among them, examples of linear dynamics on $\ell^p(\Bbb{Z})$, $L^p(\Bbb{R})$ and $L^p(\Bbb{C})$ played
important roles in this investigation (see \cite{bb09,bmp03,cs04,ge00,ls02,sa95}).

In \cite{chen11,cc11},
we characterized transitive and chaotic weighted translations on the Lebesgue space of locally compact groups,
which subsumes some results on the discrete group $\Bbb{Z}$ in \cite{ge00,sa95}.
Since then, the study of linear dynamics on locally compact groups $G$ attracted a lot of attention.
Indeed, disjoint hypercyclicity of weighted translations on $L^p(G)$ was characterized by \cite{chen172,hl16,zhang17}.
Also, the existence of hypercyclic weighted translations on $L^p(G)$ was discussed in \cite{kuchen17}. For readers interested in the hypergroup case and vector-valued version, we refer to papers \cite{chta2, chta3}. Besides, the study of linear dynamics for weighted translations on the Orlicz space $L^\Phi(G)$
was initialed by \cite{aa17,cd18} where $\Phi$ is a Young function.

From another view, Abakumov and Kuznetsova in \cite{ak17} focused on the density of translates in the weighted Lebesgue space $L_w^p(G)$, and observed some different phenomenon from that in \cite{cc11}, where $w$ is a weight on $G$. Indeed, a translation cannot have a dense orbit on the unweighted space $L^p(G)$. However, there exist hypercyclic translations on the weighted space $L_w^p(G)$ in \cite{ak17}.
We note that the weighted Orlicz space $L^\Phi_w(G)$ was recently introduced in \cite{oo15}, where Osancliol and the third author
generalized the group algebras to the weighted Orlicz algebras of locally compact groups.
Inspired by \cite{ak17,oo15}, it is nature and significant to tackle linear dynamics on the weighted Orlicz spaces $L^\Phi_w(G)$. Hence, in this note, we will study transitivity and linear chaos for translations on $L^\Phi_w(G)$.

Let $X$ be a separable Banach space and $T$  an operator on $X$.
We recall that a bounded linear  operator $T$ on $X$ is said to be {\it topologically transitive} if for any pair of non-empty open sets $U,V$ in $X$,
 there exists an $n\in\Bbb{N}$ such that $T^n(U)\cap V\neq \emptyset$.
 Furthermore, if one can find an $m\in\Bbb{N}$ such that $T^n(U)\cap V\neq \emptyset$ for all $n\geq m$, then $T$
 is {\it topologically mixing}.
 It should be noted in \cite{gpbook} that topological transitivity and hypercyclicity are equivalent on separable Banach spaces.
The latter notion arises from the invariant subset problem in analysis. A bounded linear operator $T$ is called {\it hypercyclic} if there is a
$x\in X$ whose orbit under $T$, $Orb(T,x):=\{T^nx: n\in\Bbb{N}\}$ is dense in $X$, where $T^n$ is the n-fold iteration of $T$ with $T^0=I_X$.
Moreover, the operator $T$ is called a {\it chaotic} operator if it is hypercyclic (transitive) and it possesses  a dense set of
 periodic elements.

One of the useful criterions to verify hypercyclicity for $T$ is the blow-up/collapse property
which was stated by Godefroy and Shapiro in \cite{gs91}, and called like this by Grosse-Erdmann in \cite{ge03}.
An operator $T$ satisfies the {\it blow-up/collapse property} if for every pair $U,V$ of
 non-empty open subsets of $X$, and each open neighborhood  $W$ of zero in $X$, there exists a $n\in \Bbb N$ such that
  both $T^n(U)\cap W\neq\emptyset$ and $T^n(W)\cap V\neq\emptyset$. If $T$ satisfies the blow up/collapse property, then it is hypercyclic in \cite{gs91}.

In the following, we introduce the weighted Orlicz space briefly for the further study.
A continuous, even and convex function $\Phi:{\Bbb R}\rightarrow {\Bbb R}$ is called a {\it Young function} if it satisfies $\Phi(0)=0$, $\Phi(t)>0$ for $t>0$, and $\lim_{t\rightarrow\infty}\Phi(t)=\infty$.
For a Young function $\Phi$, the complementary function $\Psi$ of $\Phi$ is given by
$$\Psi(y)=\sup\{x|y|-\Phi(x):x\geq 0\}\qquad(y\in\Bbb{R}),$$
which is also a Young function. If $\Psi$ is the complementary function of $\Phi$, then
$\Phi$ is the complementary function $\Psi$, and they satisfy the Young inequality
$$xy\leq \Phi(x)+\Psi(y)\qquad(x,y\geq0).$$

Let $G$ be a locally compact group with identity $e$ and a right Haar measure $\lambda$. Then the {\it Orlicz space} $L^\Phi(G)$ is defined by
$$L^\Phi(G)=\left\{f:G\rightarrow \Bbb{C}: \int_G\Phi(\alpha|f|)d\lambda<\infty\ \text{for some $\alpha>0$} \right\}$$
where $f$ is a Borel measurable function. Moreover, the Orlicz space is a Banach space under the Orlicz norm
defined for $f\in L^\Phi(G)$ by
$$\|f\|_{\Phi}=\sup\left\{\int_G|fv|d\lambda: \int_G\Psi(|v|)d\lambda\leq1\right\}.$$
One can also define the Luxemburg norm on $L^\Phi(G)$ by
$$N_\Phi(f)=\inf\left\{k>0:\int_G\Phi\left(\frac{|f|}{k}\right)d\lambda\leq1\right\}.$$
It is well known that these two norms are equivalent.

The Orlicz spaces are  generalization of the usual Lebesgue spaces.
The important properties of Orlicz spaces have been investigated intensely over the last several decades.
For example, Piaggio studied Orlicz spaces and the large scale geometry of Heintze groups in \cite{pi17}. Also,
the properties $(T_{L^\Phi})$ and $(F_{L^\Phi})$ for Orlicz spaces $L^\Phi$ were attained by Tanaka in \cite{ta17} recently.
For more discussions and recent works on Orlicz spaces, see \cite{cl17,ha15,rr91}.

We note that a Banach space admits a hypercyclic operator if and only if it is separable and infinite-dimensional \cite{an97,bg99}.
Hence, in this paper we assume that $G$ is second countable and $\Phi$ is $\Delta_2$-regular.
A Young function is said to be $\Delta_2$-{\it regular}  if there exist a constant $M>0$ and $t_0>0$ such that $\Phi(2t)\leq M\Phi(t)$ for $t\geq t_0$ when $G$ is compact, and $\Phi(2t)\leq M\Phi(t)$ for all $t>0$ when $G$ is non-compact in \cite{rr91}.
For example, both Young functions $\Phi$ given by
$$\Phi(t)=\frac{|t|^p}{p}\quad(1\leq p<\infty),\qquad \mbox{and}\qquad\Phi(t)=|t|^\alpha(1+|\log|t||)\quad (\alpha>1)$$
are  $\Delta_2$-regular \cite{rr91}. If $\Phi$ is $\Delta_2$-regular, then the space $C_c(G)$ of all continuous functions on $G$ with compact support is dense in $L^\Phi(G)$, and the dual space $(L^\Phi(G),\|\cdot\|_\Phi)$ is $(L^\Psi(G),N_\Psi(\cdot))$, where $\Psi$ is the complementary function of $\Phi$.

A continuous function $w:G \rightarrow (0,\infty)$ is called a {\it weight} on $G$ if
$$w(xy)\leq w(x)w(y)\qquad (x,y\in G).$$
As in \cite{oo15}, one can define the weighted Orlicz space by
$$L^\Phi_w(G):=\{f:fw\in L^\Phi(G)\}$$
endowed with the norm
$$\|f\|_{\Phi,w}:=\|fw\|_\Phi\qquad (f\in L^\Phi_w(G)),$$
which is called a {\it weighted Orlicz norm}. Then $L^\Phi_w(G)$ is a Banach space
with respect to the norm $\|\cdot\|_{\Phi,w}$. Moreover, it was showed in \cite[Lemma 2.1]{oo15}
that the space $C_c(G)$ is dense in $L^\Phi_w(G)$ if $\Phi$ is $\Delta_2$-regular.

Based on the preliminaries about the weighted Orlicz space, we next define the translation operator on $L^\Phi_w(G)$.
Let $a \in G$ and $\delta_a$ be the unit point mass at $a$. A {\it translation operator} on $G$ is
a convolution operator $T_{a}: L^\Phi_w(G)\longrightarrow L^\Phi_w(G)$ defined by
$$(T_af)(x)=(f*\delta_a)(x)= \int_{y\in G} f(xy^{-1})\delta_a(y) = f(xa^{-1}) \qquad (x\in G, f \in L^\Phi_w(G)).$$
We can also define a self-map $S_a$ on $L^\Phi_w(G)$  by
$$S_a(h) =h*\delta_{a^{-1}} \qquad (h \in L^\Phi_w(G))$$ so that
$$T_aS_a(h)=S_aT_a(h)=h \qquad (h \in L^\Phi_w(G)).$$

Since $T_a$ is generated by $a$, some elements of $G$ should be excluded from our consideration.
Indeed, it is easy to see that $T_a$ can not be hypercyclic if $a$ is a torsion element.
An element $a$ in a group $G$ is called a {\it torsion element} if
it is of finite order. In a locally compact group $G$, an element $a\in G$ is called {\it
periodic} (or {\it compact}) in \cite{rossbook} if the closed subgroup $G(a)$
generated by $a$ is compact. We call an element in
$G$ {\it aperiodic} if it is not periodic. For discrete groups,
periodic and torsion elements are identical.

\begin{lemma}
Let $G$ be a locally compact group, and  $a\in G$ be a torsion element. Let $w$ be a weight on $G$ and
$\Phi$ be a Young function. Then any translation $T_{a}:L_w^\Phi(G)\rightarrow L_w^\Phi(G)$ is not hypercyclic.
\end{lemma}
\begin{proof}
Let $a$ have order $d$, that is, $a^d=e$. Then for each $f\in L_w^\Phi(G)$, the orbit of $f$ is given by
$$\{f, f\ast\delta_a, f\ast\delta_{a^2}, \cdot\cdot\cdot, f\ast\delta_{a^{d-1}}, f\ast\delta_{a^{d}}, f\ast\delta_{a^{d+1}},\cdot\cdot\cdot\}=\{f, f\ast\delta_a, f\ast\delta_{a^2}, \cdot\cdot\cdot, f\ast\delta_{a^{d-1}}\}$$
which is a set of finite vectors, and cannot be dense in $L_w^\Phi(G)$.
Therefore $T_a$ is not hypercyclic if $a$ is torsion.


\end{proof}

In what follows, we only consider the translation $T_a$ when $a$ is aperiodic, and make use of this property of aperiodicity to obtain our results.
It was showed in \cite{cc11}, an element $a\in G$ is aperiodic if, and only if, for any compact set $K\subset G$, there exists some
$M\in\Bbb N$ such that $K\cap Ka^{\pm n}=\emptyset$ for all $n>M$.
We note that \cite{cc11} in many familiar non-discrete groups, including the additive group $\Bbb R^d$, the Heisenberg group and the affine
group, all elements except the identity are aperiodic.

\section{Chaotic conditions}

In this section, we will provide and prove the results. First, we give the characterization
for the translation operator $T_a$ on the weighted Orlicz space $L^\Phi_w(G)$ to be topologically transitive.
In particular, topological transitivity and the blow up/collapse property are equivalent in our case.

\begin{theorem}\label{transitive}
Let $G$ be a locally compact group and $a\in G$ be an aperiodic element. Let $w$ be a weight on $G$ and
$\Phi$ be a Young function. Let $T_{a}$ be a translation on $L_w^\Phi(G)$. Then the following conditions are equivalent.
\begin{enumerate}
\item[{\rm(i)}] $T_{a}$ is topologically transitive on $L_w^\Phi(G)$.
\item[{\rm(ii)}] $T_{a}$ satisfies the blow up/collapse property.
\item[{\rm(iii)}] For each compact subset $K \subseteq G$ with $\lambda(K)>0$, there exist a
sequence of Borel sets $(E_{k})$ in $K$ and a strictly increasing sequence $(n_k)\subset\Bbb{N}$ such that
$$\lim_{k \rightarrow \infty}\sup_{v\in \Omega}\int_{K\setminus{E_k}}|v(x)|w(x)d\lambda(x)=0$$
and
$$\lim_{k \rightarrow \infty}\sup_{v\in \Omega}\int_{E_k}|v(xa^{\pm n_k})|w(xa^{\pm n_k})d\lambda(x)=0$$
where $\Omega$ is the set of all Borel functions $v$ on $G$ satisfying $\int_G \Psi(|v|)d\lambda\leq 1$.
\end{enumerate}
\end{theorem}
\begin{proof}
(iii) $\Rightarrow$ (ii)
Suppose that $U$ and $V$  are nonempty open subsets of $L^\Phi_w(G)$, and $W$ is an open neighborhood of 0 in $L^\Phi_w(G)$.
Since the space $C_c(G)$ of all continuous functions on $G$ with compact support is dense in $L^\Phi_w(G)$, there are $f,g\in C_c(G)$
such that $f\in U$ and $g\in V$. We pick $\varepsilon>0$ such that the balls
$$B(f,\varepsilon):=\{h\in L_w^\Phi(G):\|h-f\|_{\Phi,w}<\varepsilon\}\subseteq U,$$
$B(g,\epsilon)\subseteq V$ and $B(0,\varepsilon)\subseteq W$. Let $K$ be the union of supports of $f$ and $g$. Assume that $(E_k)$ and $(n_k)$
satisfy the hypothesis with respect to the compact set $K\subseteq G$. Then there is $k>0$ such that
$$\|f\|_{\infty}\sup_{v\in \Omega}\int_{K\setminus{E_k}}|v(x)|w(x)d\lambda(x)<\varepsilon,$$
$$\|f\|_{\infty}\sup_{v\in \Omega}\int_{E_k}|v(xa^{n_k})|w(xa^{n_k})d\lambda(x)<\varepsilon$$
and
$$\|f\|_{\infty}\sup_{v\in \Omega}\int_{E_k}|v(xa^{-n_k})|w(xa^{-n_k})d\lambda(x)<\varepsilon.$$
Therefore
\begin{eqnarray*}
\|T_{a}^{n_k}(f\chi_{E_k})\|_{\Phi,w}&=&\sup_{v\in\Omega}\int_G|T_a^{n_k}(f\chi_{E_k})(x)v(x)|w(x)d\lambda(x)\\
&=&\sup_{v\in\Omega}\int_G|f(xa^{-n_k})\chi_{E_k}(xa^{-n_k})v(x)|w(x)d\lambda(x)\\
&=&\sup_{v\in\Omega}\int_G|f(x)\chi_{E_k}(x)v(xa^{n_k})|w(xa^{n_k})d\lambda(x)\\
&\leq& \|f\|_{\infty} \sup_{v\in\Omega}\int_{E_k}|v(xa^{n_k})|w(xa^{n_k})d\lambda(x)<\varepsilon.
\end{eqnarray*}
Hence $T_{a}^{n_k}(f\chi_{E_k})\in W$. Also,
\begin{eqnarray*}
\|f-f\chi_{E_k}\|_{\Phi,w}&=&\sup_{v\in\Omega}\int_G|f(x)-f(x)\chi_{E_k}(x)||v(x)|w(x)d\lambda(x)\\
&=&\sup_{v\in\Omega}\int_G|f(x)\chi_{K\setminus E_k}(x)||v(x)|w(x)d\lambda(x)\\
&=&\sup_{v\in\Omega}\int_{K\setminus E_k}|f(x)||v(x)|w(x)d\lambda(x)\\
&\leq& \|f\|_{\infty}\int_{K\setminus E_k}|v(x)|w(x)d\lambda(x)<\varepsilon,
\end{eqnarray*}
which says $f\chi_{E_k}\in U$. Hence, $T_{a}^{n_k}(f\chi_{E_k})\in T_{a}^{n_k}(U)\cap W$.
By applying the similar arguments for $S_{a}$ and $g$, one can obtain that $S_{a}^{n_k}(g\chi_{E_k})\in  S_{a}^{n_k}(V)\cap W$, that is,
$g\chi_{E_k}\in  V\cap T_{a}^{n_k}(W)$. Combing all these, $T_{a}$ satisfies the blow up/collapse property.

(i) $\Rightarrow$ (iii).
By the assumptions of topological transitivity and aperiodicity of $a$,
there exist $f\in L_w^\Phi(G)$ and some $m\in \Bbb{N}$ such that
$K\cap Ka^{\pm m}=\emptyset$,
$$\|f-\chi_{K}\|_{\Phi,w} < \varepsilon^2 \quad \mbox{and} \quad \|T_{a}^{m}f-\chi_{K}\|_{\Phi,w}< \varepsilon^2.$$
Let $$A=\{x\in K:|f(x)-1|\geq\varepsilon\}.$$
Then
$$|f(x)|>1-\varepsilon \qquad (x\in K\setminus A)\qquad \mbox{and}\qquad \sup_{v\in \Omega}\int_A|v(x)|w(x)d\lambda(x)<\varepsilon$$
by
\begin{eqnarray*}
\varepsilon^2 &>& \|f-\chi_{K}\|_{\Phi,w}\\
&=& \sup_{v\in \Omega}\int_G|f(x)-\chi_{K}(x)||v(x)|w(x)d\lambda(x)\\
&\geq& \sup_{v\in \Omega}\int_A|f(x)-1||v(x)|w(x)d\lambda(x)\\
&>& \sup_{v\in \Omega}\int_A \varepsilon |v(x)|w(x)d\lambda(x).
\end{eqnarray*}
Let $$B_m=\{x\in K:|T_{a}^{m}f(x)-1|\geq\varepsilon\}.$$
Then
$$|T_{a}^{m}f(x)|>1-\varepsilon \qquad ( x\in K\setminus B_{m})\qquad \mbox{and}\qquad \sup_{v\in \Omega}\int_{B_m}|v(x)|w(x)d\lambda(x)<\varepsilon$$
by the following estimate
\begin{eqnarray*}
\varepsilon^2 &>& \|T_{a}^{m}f-\chi_{K}\|_{\Phi,w}\\
&=& \sup_{v\in \Omega}\int_G|(T_{a}^{m}f)(x)-\chi_{K}(x)||v(x)|w(x)d\lambda(x)\\
&\geq& \sup_{v\in \Omega}\int_{B_m}|(T_{a}^{m}f)(x)-1||v(x)|w(x)d\lambda(x)\\
&>& \sup_{v\in \Omega}\int_{B_m}\varepsilon |v(x)|w(x)d\lambda(x).
\end{eqnarray*}

Let $E_m=K\setminus (A\cup B_m)$. Then by $K\cap Ka^{\pm m}=\emptyset$, we have
\begin{eqnarray*}
\varepsilon^2 &>& \|T_{a}^{m}f-\chi_{K}\|_{\Phi,w}\\
&=& \sup_{v\in \Omega}\int_G|(T_{a}^{m}f)(x)-\chi_{K}(x)||v(x)|w(x)d\lambda(x)\\
&=& \sup_{v\in \Omega}\int_G|f(xa^{-m})-\chi_{K}(x)||v(x)|w(x)d\lambda(x)\\
&=& \sup_{v\in \Omega}\int_G|f(x)-\chi_{K}(xa^m)||v(xa^m)|w(xa^m)d\lambda(x)\\
&\geq& \sup_{v\in \Omega}\int_{E_m}|f(x)-\chi_{K}(xa^m)||v(xa^m)|w(xa^m)d\lambda(x)\\
&=& \sup_{v\in \Omega}\int_{E_m}|f(x)||v(xa^m)|w(xa^m)d\lambda(x)\\
&>& \sup_{v\in \Omega}\int_{E_m}(1-\varepsilon)|v(xa^m)|w(xa^m)d\lambda(x)
\end{eqnarray*}
Hence,
$$\sup_{v\in \Omega}\int_{E_m}|v(xa^m)|w(xa^m)d\lambda(x)<\frac{\varepsilon^2}{1-\varepsilon}.$$
Similarly,
$$\sup_{v\in \Omega}\int_{E_m}|v(xa^{-m})|w(xa^{-m})d\lambda(x)<\frac{\varepsilon^2}{1-\varepsilon}$$
by
\begin{eqnarray*}
\varepsilon^2 &>& \|f-\chi_{K}\|_{\Phi,w}\\
&=& \sup_{v\in \Omega}\int_G|S_a^m(T_{a}^{m}f)(x)-\chi_{K}(x)||v(x)|w(x)d\lambda(x)\\
&=& \sup_{v\in \Omega}\int_G|(T_{a}^{m}f)(xa^{m})-\chi_{K}(x)||v(x)|w(x)d\lambda(x)\\
&=& \sup_{v\in \Omega}\int_G|(T_{a}^{m}f)(x)-\chi_{K}(xa^{-m})||v(xa^{-m})|w(xa^{-m})d\lambda(x)\\
&\geq& \sup_{v\in \Omega}\int_{E_m}|(T_{a}^{m}f)(x)-\chi_{K}(xa^{-m})||v(xa^{-m})|w(xa^{-m})d\lambda(x)\\
&=& \sup_{v\in \Omega}\int_{E_m}|(T_{a}^{m}f)(x)||v(xa^{-m})|w(xa^{-m})d\lambda(x)\\
&>& \sup_{v\in \Omega}\int_{E_m}(1-\varepsilon)|v(xa^{-m})|w(xa^{-m})d\lambda(x).
\end{eqnarray*}

Also, we have
\begin{eqnarray*}
&& \sup_{v\in \Omega}\int_{K\setminus{E_m}}|v(x)|w(x)d\lambda(x)\\
&=& \sup_{v\in \Omega}\int_{A\cup B_m}|v(x)|w(x)d\lambda(x)\\
&\leq& \sup_{v\in \Omega}\int_{A}|v(x)|w(x)d\lambda(x)
+ \sup_{v\in \Omega}\int_{B_m}|v(x)|w(x)d\lambda(x)\\
&<&\varepsilon+\varepsilon=2\varepsilon.
\end{eqnarray*}

Combining all these, condition (iii) follows.
\end{proof}

\begin{example}
Let $G=\Bbb{R}$ and $a=5$. Let $w$ be a weight on $\Bbb{R}$.
Then the translation $T_{5}$ on $L^\Phi_w(\Bbb{R})$ is defined by
$$T_{5}f(x)=f(x-5)\qquad(f\in L^\Phi_w(\Bbb{R})).$$
By Theorem \ref{transitive}, $T_{5}$ is topologically transitive if given a compact subset $K \subseteq \Bbb{R}$, there exist a
sequence of Borel sets $(E_{k})$ in $K$ and a strictly increasing sequence $(n_k)\subset\Bbb{N}$ such that
$$\lim_{k \rightarrow \infty}\sup_{v\in \Omega}\int_{K\setminus{E_k}}|v(x)|w(x)dx=0$$
and
$$\lim_{k \rightarrow \infty}\sup_{v\in \Omega}\int_{E_k}|v(x\pm 5n_k)|w(x\pm 5n_k)dx=0$$
where $\Omega$ is the set of all Borel functions $v$ on $\Bbb{R}$ satisfying $\int_{\Bbb{R}}\Psi(|v(x)|)dx\leq 1$.
\end{example}

By strengthening the condition (iii) of Theorem \ref{transitive}, one can give
a sufficient and necessary condition for $T_a$ to be topologically mixing.

\begin{corollary}
Let $G$ be a locally compact group and  $a\in G$ be an aperiodic element. Let $w$ be a weight on $G$ and
$\Phi$ be a Young function. Let $T_{a}$ be a translation on $L_w^\Phi(G)$. Then the following conditions are equivalent.
\begin{enumerate}
\item[{\rm(i)}] $T_{a}$ is topologically mixing on $L_w^\Phi(G)$.
\item[{\rm(ii)}] For each compact subset $K \subseteq G$ with $\lambda(K)>0$, there exists a
sequence of Borel sets $(E_{n})$ in $K$ such that
$$\lim_{n \rightarrow \infty}\sup_{v\in \Omega}\int_{K\setminus{E_n}}|v(x)|w(x)d\lambda(x)=0$$
and
$$\lim_{n \rightarrow \infty}\sup_{v\in \Omega}\int_{E_n}|v(xa^{\pm n})|w(xa^{\pm n})d\lambda(x)=0$$
where $\Omega$ is the set of all Borel functions $v$ on $G$ satisfying $\int_G \Psi(|v|)d\lambda\leq 1$.
\end{enumerate}
\end{corollary}
\begin{proof}
The proof is similar to that of Theorem \ref{transitive} by using the full sequence $(n)$ instead of subsequence $(n_k)$.
\end{proof}

Based on Theorem \ref{transitive}, we end the paper by characterizing chaotic translations on the weighted Orlicz space $L_w^\Phi(G)$, and showing that
the dense set of periodic elements implies transitivity automatically.

\begin{theorem}\label{chaos}
Let $G$ be a locally compact group and $a\in G$ be an aperiodic element. Let $w$ be a weight on $G$ and
$\Phi$ be a Young function. Let $T_{a}$ be a translation on $L_w^\Phi(G)$ and ${\cal P}(T_{a})$ be the set of periodic elements of $T_a$.
Then the following conditions are equivalent.
\begin{enumerate}
\item[{\rm(i)}] $T_{a}$ is chaotic on $L_w^\Phi(G)$.
\item[{\rm(ii)}] ${\cal P}(T_{a})$ is dense in $L_w^\Phi(G)$.
\item[{\rm(iii)}] For each compact subset $K \subseteq G$ with $\lambda(K)>0$, there exist a
sequence of Borel sets $(E_{k})$ in $K$ and a strictly increasing sequence $(n_k)\subset\Bbb{N}$ such that
$$\lim_{k \rightarrow \infty}\sup_{v\in \Omega}\int_{K\setminus{E_k}}|v(x)|w(x)d\lambda(x)=0$$
and
$$\lim_{k \rightarrow \infty}\sup_{v\in \Omega}\left(\sum_{l=1}^{\infty}\int_{E_k}|v(xa^{ln_k})|w(xa^{ln_k})d\lambda(x)
+\sum_{l=1}^{\infty}\int_{E_k}|v(xa^{-ln_k})|w(xa^{-ln_k})d\lambda(x)\right)=0$$
where $\Omega$ is the set of all Borel functions $v$ on $G$ satisfying $\int_G \Psi(|v|)d\lambda\leq 1$.
\end{enumerate}
\end{theorem}
\begin{proof}
We will show (ii) $\Rightarrow$ (iii), and (iii) $\Rightarrow$ (i).

(ii) $\Rightarrow$ (iii). Let $K \subseteq G$ be a compact set with $\lambda(K) >0$.
Since $a$ is aperiodic, there exists $M\in\Bbb{N}$ such that $K\cap Ka^{\pm m}=\emptyset$ for all $m> M$. Let $\chi_K\in L_w^\Phi(G)$
be the characteristic function of $K$. By the density of ${\cal P}(T_a)$, we can find a sequence $(f_k)$ of periodic points of $T_{a}$ satisfying
$\|f_k-\chi_{K}\|_{\Phi,w}<\frac{1}{4^k}$, and a sequence $(n_k)\subset \Bbb{N}$ such that $T_{a}^{n_k}f_k=f_k=S_{a}^{n_k}f_k$, where we may assume $n_{k+1}>n_k> M$.
Therefore, $Ka^{rn_k}\cap Ka^{sn_k}=\emptyset$ for all $r,s\in \Bbb{Z}$ with $r\neq s$.

Let $A_k=\{x\in K:|f_k(x)-1|\geq\frac{1}{2^k}\}$, and let $E_k=K\setminus A_k$. Then
$$|f_k(x)|>1-\frac{1}{2^k} \qquad ( x\in E_k),$$
and by the similar argument as in the proof Theorem \ref{transitive}, one has
$$ \sup_{v\in \Omega}\int_{A_k}|v(x)|w(x)d\lambda(x)<\frac{1}{2^k},$$
that is,
$$\sup_{v\in \Omega}\int_{K\setminus{E_k}}|v(x)|w(x)d\lambda(x)<\frac{1}{2^k}.$$
On the other hand, by $Ka^{rn_k}\cap Ka^{sn_k}=\emptyset$ and the right invariance of the Haar measure $\lambda$,
\begin{eqnarray*}
&&\frac{1}{4^k} > \|f_k-\chi_{K}\|_{\Phi,w}\\
&=& \sup_{v\in \Omega}\int_G|f_k(x)-\chi_{K}(x)||v(x)|w(x)d\lambda(x)\\
&\geq& \sup_{v\in \Omega}\int_{G\setminus K}|f_k(x)||v(x)|w(x)d\lambda(x)\\
&\geq& \sup_{v\in \Omega}\left(\sum_{l=1}^{\infty}\int_{Ka^{ln_k}}|f_k(x)||v(x)|w(x)d\lambda(x)
+\sum_{l=1}^{\infty}\int_{Ka^{-ln_k}}|f_k(x)||v(x)|w(x)d\lambda(x)\right)\\
&=& \sup_{v\in \Omega}\left(\sum_{l=1}^{\infty}\int_{Ka^{ln_k}}|(T_a^{ln_k}f_k)(x)||v(x)|w(x)d\lambda(x)
+\sum_{l=1}^{\infty}\int_{Ka^{-ln_k}}|(S_a^{ln_k}f_k)(x)||v(x)|w(x)d\lambda(x)\right)\\
&=& \sup_{v\in \Omega}\left(\sum_{l=1}^{\infty}\int_{Ka^{ln_k}}|f_k(xa^{-ln_k})||v(x)|w(x)d\lambda(x)
+\sum_{l=1}^{\infty}\int_{Ka^{-ln_k}}|f_k(xa^{ln_k})||v(x)|w(x)d\lambda(x)\right)\\
&=& \sup_{v\in \Omega}\left(\sum_{l=1}^{\infty}\int_{K}|f_k(x)||v(xa^{ln_k})|w(xa^{ln_k})d\lambda(x)
+\sum_{l=1}^{\infty}\int_{K}|f_k(x)||v(xa^{-ln_k})|w(xa^{-ln_k})d\lambda(x)\right)\\
&\geq& \sup_{v\in \Omega}\left(\sum_{l=1}^{\infty}\int_{E_k}|f_k(x)||v(xa^{ln_k})|w(xa^{ln_k})d\lambda(x)
+\sum_{l=1}^{\infty}\int_{E_k}|f_k(x)||v(xa^{-ln_k})|w(xa^{-ln_k})d\lambda(x)\right)\\
&>& (1-\frac{1}{2^k})\sup_{v\in \Omega}\left(\sum_{l=1}^{\infty}\int_{E_k}|v(xa^{ln_k})|w(xa^{ln_k})d\lambda(x)
+\sum_{l=1}^{\infty}\int_{E_k}|v(xa^{-ln_k})|w(xa^{-ln_k})d\lambda(x)\right).
\end{eqnarray*}
Hence, the condition (iii) is obtained.

(iii) $\Rightarrow$ (i). By Theorem \ref{transitive}, condition (iii) implies that $T_a$ is topologically transitive.
Here, we will only show that ${\cal P}(T_a)$ is dense in $L_w^\Phi(G)$.
Let $f\in C_c(G)$ with compact support $K\subseteq G$. Then there exist a sequence of Borel sets $(E_k)$ in $K$, and a sequence $(n_k)\subseteq\mathbb N$
such that $Ka^{rn_k}\cap Ka^{sn_k}=\emptyset$,
$$\sup_{v\in \Omega}\int_{K\setminus{E_k}}|v(x)|w(x)d\lambda(x)<\frac{1}{2^k}$$
and
$$\sup_{v\in \Omega}\left(\sum_{l=1}^{\infty}\int_{E_k}|v(xa^{ln_k})|w(xa^{ln_k})d\lambda(x)+
\sum_{l=1}^{\infty}\int_{E_k}|v(xa^{-ln_k})|w(xa^{-ln_k})d\lambda(x)\right)<\frac{1}{2^k}.$$
Let
$$v_k=f\chi_{E_k}+\sum_{l=1}^{\infty}T_a^{ln_k}(f\chi_{E_k})+\sum_{l=1}^{\infty}S_a^{ln_k}(f\chi_{E_k}).$$
Then
\begin{eqnarray*}
T_a^{n_k}v_k&=&T_a^{n_k}(f\chi_{E_k})+\sum_{l=1}^{\infty}T_a^{n_k}T_a^{ln_k}(f\chi_{E_k})
+\sum_{l=1}^{\infty}T_a^{n_k}S_a^{ln_k}(f\chi_{E_k})\\
&=&\sum_{l=1}^{\infty}T_a^{ln_k}(f\chi_{E_k})+f\chi_{E_k}+\sum_{l=1}^{\infty}S_a^{ln_k}(f\chi_{E_k})=v_k
\end{eqnarray*}
which implies that $v_k\in {\cal P}(T_a)$ for each $k$. Moreover,
\begin{eqnarray*}
\|f-f\chi_{E_k}\|_{\Phi,w}
&=& \sup_{v\in \Omega}\int_G|f(x)-f(x)\chi_{E_k}(x)||v(x)|w(x)d\lambda(x)\\
&=& \sup_{v\in \Omega}\int_{K\setminus E_k}|f(x)||v(x)|w(x)d\lambda(x)< \frac{\|f\|_{\infty}}{2^k}.
\end{eqnarray*}
Also, by $Ka^{rn_k}\cap Ka^{sn_k}=\emptyset$ and the right invariance of $\lambda$, we have
\begin{eqnarray*}
&&\left\|\sum_{l=1}^{\infty}T_a^{ln_k}(f\chi_{E_k})+\sum_{l=1}^{\infty}S_a^{ln_k}(f\chi_{E_k})\right\|_{\Phi,w}\\
&=& \sup_{v\in \Omega}\int_G\left|\sum_{l=1}^{\infty}T_a^{ln_k}(f\chi_{E_k})(x)+\sum_{l=1}^{\infty}S_a^{ln_k}(f\chi_{E_k})\right||v(x)|w(x)d\lambda(x)\\
&=& \sup_{v\in \Omega}\int_G\left|\sum_{l=1}^{\infty}f(xa^{-ln_k})\chi_{E_k}(xa^{-ln_k})
+\sum_{l=1}^{\infty}f(xa^{ln_k})\chi_{E_k}(xa^{ln_k})\right||v(x)|w(x)d\lambda(x)\\
&=& \sup_{v\in \Omega}\Big(\sum_{l=1}^{\infty}\int_G|f(xa^{-ln_k})\chi_{E_k}(xa^{-ln_k})||v(x)|w(x)d\lambda(x)\\
&&\qquad+\sum_{l=1}^{\infty}\int_G|f(xa^{ln_k})\chi_{E_k}(xa^{ln_k})||v(x)|w(x)d\lambda(x)\Big)\\
&=& \sup_{v\in \Omega}\Big(\sum_{l=1}^{\infty}\int_G|f(x)\chi_{E_k}(x)||v(xa^{ln_k})|w(xa^{ln_k})d\lambda(x)\\
&&\qquad+\sum_{l=1}^{\infty}\int_G|f(x)\chi_{E_k}(x)||v(xa^{-ln_k})|w(xa^{-ln_k})d\lambda(x)\Big)\\
&\leq& \|f\|_{\infty} \sup_{v\in \Omega}\Big(\sum_{l=1}^{\infty}\int_{E_k}|v(xa^{ln_k})|w(xa^{ln_k})d\lambda(x)
+\sum_{l=1}^{\infty}\int_{E_k}|v(xa^{-ln_k})|w(xa^{-ln_k})d\lambda(x)\Big)\\
&<&\frac{\|f\|_{\infty}}{2^k}.
\end{eqnarray*}
Therefore,
$$\|v_k-f\|_{\Phi,w}\leq\|f\chi_{E_k}-f\|_{\Phi,w}
+\left\|\sum_{l=1}^{\infty}T_a^{ln_k}(f\chi_{E_k})+\sum_{l=1}^{\infty}S_a^{ln_k}(f\chi_{E_k})\right\|_{\Phi,w}\rightarrow 0$$
as $k\rightarrow\infty$.
Hence, the set ${\cal P}(T_a)$ is dense in $L_w^\Phi(G)$.
\end{proof}

\begin{example}
Let $$G=\Bbb{H}:=\left\{\left(\begin{matrix}
1 & x & z\\
0 & 1 & y\\
0 & 0 & 1
\end{matrix}\right):x,y,z\in \Bbb{R}\right\}$$
be the Heisenberg group. For
convenience, an element in $\Bbb{H}$ is written as $(x,y,z)$.

Let $a=(1,0,3)$ and $w$ be a weight on $\Bbb{H}$.
Then $a^{-1}=(-1,0,-3)$ and the translation $T_{(1,0,3)}$
on $L^\Phi_w(\Bbb{H})$ is given by
$$T_{(1,0,3)}f(x,y,z)=f(x-1,y,z-3)\qquad(f\in L^\Phi_w(\Bbb{H})).$$
By Theorem \ref{chaos}, the operator $T_{(1,0,3)}$
is chaotic if given a compact subset $K \subseteq \Bbb{H}$, there exist a
sequence of Borel sets $(E_{k})$ in $K$ and a strictly increasing sequence $(n_k)\subset\Bbb{N}$ such that
$$\lim_{k \rightarrow \infty}\sup_{v\in \Omega}\int_{K\setminus{E_k}}|v(x,y,z)|w(x,y,z)dxdydz=0$$
and
\begin{eqnarray*}
&&
\lim_{k \rightarrow \infty}\sup_{v\in \Omega}\Big(\sum_{l=1}^{\infty}\int_{E_k}|v(x+ln_k,y,z+3ln_k)|w(x+ln_k,y,z+3ln_k)dxdydz\\
&&+\sum_{l=1}^{\infty}\int_{E_k}|v(x-ln_k,y,z-3ln_k)|w(x-ln_k,y,z-3ln_k)dxdydz\Big)=0
\end{eqnarray*}
where $\Omega$ is the set of all Borel functions $v$ on $\Bbb{H}$ satisfying
$$\int_{\Bbb{H}} \Psi(|v(x,y,z)|)dxdydz\leq 1.$$
\end{example}

\vspace{.1in}
\end{document}